\documentclass[a4paper]{article}

\usepackage{etoolbox}
\usepackage{expl3}
\usepackage{xifthen}
\usepackage{xparse}

\usepackage[utf8]{inputenc}
\usepackage[T1]{fontenc}
\usepackage{geometry}
\usepackage{authblk} % for affiliations
\usepackage[british]{babel}
\usepackage{enumitem} % for enumerations
\usepackage{xcolor} % for colour specifications
\usepackage{lineno}  % to add line numbers for correction purposes
\usepackage{lmodern}  % font type
\usepackage[defaultlines=2]{nowidow}  % specifies the minimal number of lines to keep before and after page break
\usepackage[hyphens]{url} % for better line breaks for urls
\usepackage{hyperref} % for clickable links
\usepackage{orcidlink} % for ORCID specification
\usepackage{microtype}  % for micro-typographic adjustments
\usepackage{csquotes} % for quotations marks according to the used language
\hypersetup{
	colorlinks = true,  % Links farbig (statt Kasten)
	linkcolor = stdred,  % Links (auf interne Referenzen)
	citecolor = stdgreen,  % Quellenverweis
	urlcolor = stdblue,  % Web, Mail
}

\newlength\oldparindent
\renewcommand{\subparagraph}[1]{\medskip\noindent\textbf{#1}}

\setlist{noitemsep,labelwidth=*,leftmargin=*,align=left}
\setlist[1]{labelindent=1em}
\setlist[enumerate,1]{label=(\alph*)}
\setlist[enumerate,2]{label=(\arabic*),ref=\theenumi~(\arabic*)}
\setlist[description]{font=\normalfont,leftmargin=!}

\newlist{enumprop}{enumerate}{2}  % for properties
\setlist[enumprop,1]{label=(\alph*),ref=(\alph*)}
\setlist[enumprop,2]{label=(\arabic*),ref=\theenumpropi~(\arabic*)}
\newlist{enumcond}{enumerate}{2}  % for conditions
\setlist[enumcond,1]{label=(\roman*),ref=(\roman*)}
\setlist[enumcond,2]{label=(\arabic*),ref=\theenumcondi~(\arabic*)}
\newlist{enumstep}{enumerate}{2}  % for steps
\setlist[enumstep,1]{label=(\arabic*),ref=(\arabic*)}
\setlist[enumstep,2]{label=(\alph*),ref=\theenumstepi~(\alph*)}
\newlist{enumpart}{enumerate}{2}  % for parts / statements
\setlist[enumpart,1]{label=(\alph*),ref=(\alph*)}
\setlist[enumpart,2]{label=(\arabic*),ref=\theenumparti~(\arabic*)}
\newlist{enumop}{enumerate}{2}  % for operations
\setlist[enumop,1]{label=(\roman*),ref=(\roman*)}
\setlist[enumop,2]{label=(\arabic*),ref=\theenumopi~(\arabic*)}
\newlist{enummod}{enumerate}{2}  % for modifications
\setlist[enummod,1]{label=(\arabic*),ref=(\arabic*)}
\setlist[enummod,2]{label=(\alph*),ref=\theenummodi~(\alph*)}

\usepackage{graphicx}
\usepackage{caption}
\usepackage{subcaption}
\usepackage{rotating}
\usepackage{multirow} % for unifying rows in a table
\definecolor{stdred}{RGB}{255,31,91}
\definecolor{stdgreen}{RGB}{0,205,108}
\definecolor{stdblue}{RGB}{0,152,222}
\definecolor{stdpurple}{RGB}{175,88,186}
\definecolor{stdyellow}{RGB}{255,198,30}
\definecolor{stdorange}{RGB}{242,133,34}
\definecolor{stdbrown}{RGB}{166,118,29}
\definecolor{stdgrey}{RGB}{160,177,186} % the recommended grey, but, at least with the local printer quality, it looks mediocre at best
\definecolor{grey}{RGB}{127,127,127} % a darker grey
\usepackage{pdfpages}
\usepackage{array}

\newcolumntype{L}[2]{>{#1\raggedright\let\newline\\\arraybackslash\hspace{0pt}}m{#2}}
\newcolumntype{C}[2]{>{#1\centering\let\newline\\\arraybackslash\hspace{0pt}}m{#2}}
\newcolumntype{R}[2]{>{#1\raggedleft\let\newline\\\arraybackslash\hspace{0pt}}m{#2}}

\usepackage{amsmath}
\usepackage{framed}
\usepackage[framed, hyperref, thmmarks, amsmath]{ntheorem}
\usepackage{mathrsfs} % provides \mathscr if required

\usepackage{amsfonts}
\usepackage{amssymb}
\usepackage{mathtools}
\usepackage{extarrows}
\usepackage{old-arrows} % to make stuff like \overleftarrow smaller
\usepackage{colonequals}

\theoremstyle{change}
\theoremheaderfont{\normalfont\bfseries}
\theorembodyfont{\normalfont}
\theorempreskip{\topsep}
\theorempostskip{\topsep}
\theoremseparator{.}
\theoremsymbol{\ensuremath{\triangleleft}}
\newtheorem{definition}{Definition}[section]

\newtheorem{lemma}[definition]{Lemma}

\newtheorem{theorem}[definition]{Theorem}
\newtheorem{corollary}[definition]{Corollary}

\newtheorem{observation}[definition]{Observation}

\newtheorem{assumption}[definition]{Assumption}

\theoremstyle{plain}

\renewtheorem*{claim*}{Claim}
\makeatletter
\@addtoreset{claim}{definition} % reset the claim counter each time the definition counter is incremented
\makeatother

\theoremstyle{changebreak}
\theoremheaderfont{\normalfont\bfseries}
\theorembodyfont{\normalfont}
\theorempreskip{\topsep}
\theorempostskip{\topsep}
\theoremseparator{.}
\theoremsymbol{\ensuremath{\triangleleft}}

\theoremstyle{break}

\makeatletter
\@addtoreset{claim}{definition} % reset the claim counter each time the definition counter is incremented
\makeatother

\theoremstyle{change}
\theoremheaderfont{\normalfont\bfseries}
\theorembodyfont{\normalfont\slshape}
\theorempreskip{\topsep}
\theorempostskip{\topsep}
\theoremseparator{.}
\theoremsymbol{\ensuremath{\triangleleft}}

\newtheorem{theoremcite}[definition]{Theorem}

\theoremstyle{nonumberplain}
\theoremheaderfont{\itshape}
\theorembodyfont{\upshape}
\theorempreskip{\topsep}
\theorempostskip{\topsep}
\theoremseparator{.}
\theoremsymbol{}

\theoremsymbol{\ensuremath{\square}}
\newtheorem{proof}{Proof}
\theoremsymbol{\ensuremath{\diamond}}

\theoremstyle{nonumberbreak}
\theoremheaderfont{\itshape}
\theorembodyfont{\upshape}
\theorempreskip{\topsep}
\theorempostskip{\topsep}
\theoremseparator{.}
\theoremsymbol{}

\theoremsymbol{\ensuremath{\square}}

\theoremsymbol{\ensuremath{\diamond}}

\ExplSyntaxOn
\NewDocumentCommand{\instring}{mmmm}
{
	\oleks_instring:nnnn { #1 } { #2 } { #3 } { #4 }
}
\tl_new:N \l__oleks_instring_test_tl
\cs_new_protected:Nn \oleks_instring:nnnn
{
	\tl_set:Nn \l__oleks_instring_test_tl { #1 }
	\regex_match:nnTF { \u{l__oleks_instring_test_tl} } { #2 } { #3 } { #4 }
}
\ExplSyntaxOff

\newcommand{\RR}{\ensuremath{\mathbb{R}}}
\newcommand{\ZZ}{\ensuremath{\mathbb{Z}}}

\def\setdelimiter{\colon\,}
\def\replacecolonbydelimiter#1:#2\relax{#1\setdelimiter#2}
\def\setcontents#1{\replacecolonbydelimiter#1\relax}
\NewDocumentCommand{\Set}{m}{
	\instring{:}{#1}{% if a condition is present
		\ensuremath{\left\{\setcontents{#1}\right\}}%
	}{% if all elements are listed
		\ensuremath{\left\{#1\right\}}%
	}%
}

\renewcommand{\emptyset}{\varnothing}

\newcommand{\T}{^{\mathsf T}}
\newcommand{\vtr}[1]{\begin{pmatrix} #1\vtrarg}
	\newcommand{\vtrarg}{\@ifnextchar\bgroup{\vtrnextarg}{\end{pmatrix}}}
\newcommand{\vtrnextarg}[1]{ \\#1\@ifnextchar\bgroup{\vtrnextarg}{\end{pmatrix}}}

\newcommand{\vtrtarg}{\@ifnextchar\bgroup{\vtrtnextarg}{\right)}}
\newcommand{\vtrtnextarg}[1]{ \;#1\@ifnextchar\bgroup{\vtrtnextarg}{\right)^\T}}

\allowdisplaybreaks[1]

\usepackage[algosection,ruled,resetcount,noend]{algorithm2e}

\makeatletter
\patchcmd{\@algocf@start}% <cmd>
{-1.5em}% <search>
{0pt}% <replace>
{}{}% <success><failure>
\makeatother

\renewcommand{\SetProgSty}[1]{\renewcommand{\ProgSty}[1]{\textnormal{\csname#1\endcsname{##1}}\unskip}}
\SetProgSty{}
\SetFuncSty{textsc}
\SetFuncArgSty{}
\SetArgSty{}
\SetCommentSty{emph}
\DontPrintSemicolon

\SetKwComment{tcp}{\textup{\texttt{\#}} }{}
\SetKwComment{tcc}{''' }{ '''}
\SetKwProg{Def}{def}{:}{}

\SetKw{Continue}{continue}
\SetKw{Break}{break}

\usepackage{tikz}
\usetikzlibrary{arrows,petri,topaths}
\usetikzlibrary{decorations,patterns}
\usetikzlibrary{backgrounds}
\usetikzlibrary{fit}
\usetikzlibrary{graphs}
\usetikzlibrary{calc,positioning}
\usetikzlibrary{fadings}
\usetikzlibrary{intersections}
\usetikzlibrary{scopes}
\usetikzlibrary{shapes}

\tikzset{node/.style={circle,draw=black,fill=black,scale=0.5}}
\tikzset{edge/.style={very thick}}
\tikzset{arc/.style={edge,->}}
\tikzset{labelnode/.style={anchor=base}}
\tikzset{highlightpath/.style={draw,line width=8pt,red!20,line cap=round,line join=round}}
\tikzset{>=latex'}

\tikzset{
	dot diameter/.store in=\dot@diameter,
	dot diameter=2pt,
	dot spacing/.store in=\dot@spacing,
	dot spacing=4pt,
	dots/.style={
		line width=\dot@diameter,
		line cap=round,
		dash pattern=on 0pt off \dot@spacing
	}
}
\tikzset{dotted/.style={very thick,dot diameter=2pt,dot spacing=4pt,dots}}

\usepackage{pgfplots}
\pgfplotsset{compat=1.18}

\usepackage[capitalise,noabbrev]{cleveref}

\crefname{enumpropi}{Property}{Properties}
\crefname{enumpropii}{Property}{Properties}
\crefname{enumcondi}{Condition}{Conditions}
\crefname{enumcondii}{Condition}{Conditions}
\crefname{enumstepi}{Step}{Steps}
\crefname{enumstepii}{Step}{Steps}
\crefname{enumparti}{Part}{Parts}
\crefname{enumpartii}{Part}{Parts}
\crefname{enumopi}{Operation}{Operations}
\crefname{enumopii}{Operation}{Operations}
\crefname{enummodi}{Modification}{Modifications}
\crefname{enummodii}{Modification}{Modifications}

\crefname{conjecture}{Conjecture}{Conjectures}
\crefname{convention}{Convention}{Conventions}
\crefname{notation}{Notation}{Notations}
\crefname{observation}{Observation}{Observations}
\crefname{exercise}{Exercise}{Exercises}
\crefname{note}{Note}{Notes}
\crefname{claim}{Claim}{Claims}
\crefname{problem}{Problem}{Problems}
\crefname{assumption}{Assumption}{Assumptions}
\crefname{namedtheorem}{Theorem}{Theorems}
\crefname{namedtheoremcite}{Theorem}{Theorems}
\crefname{namedlemma}{Lemma}{Lemmas}
\crefname{namedlemmacite}{Lemma}{Lemmas}
\crefname{namedconjecture}{Conjecture}{Conjectures}

\crefalias{definitionenum}{definition}
\crefalias{conjectureenum}{conjecture}
\crefalias{lemmaenum}{lemma}
\crefalias{exampleenum}{example}
\crefalias{theoremenum}{theorem}
\crefalias{corollaryenum}{corollary}
\crefalias{conventionenum}{convention}
\crefalias{notationenum}{notation}
\crefalias{remarkenum}{remark}
\crefalias{observationenum}{observation}
\crefalias{exerciseenum}{exercise}
\crefalias{lemmacite}{lemma}
\crefalias{theoremcite}{theorem}
\crefalias{remarkcite}{remark}
\crefalias{notecited}{note}
\crefalias{noteenum}{note}
\crefalias{claimenum}{note}

\usepackage[bibencoding=utf8,style=alphabetic,backend=biber]{biblatex}
\DefineBibliographyExtras{british}{} % enable Oxford comma in british

\renewbibmacro*{in:}{} % to remove the "In:" before the journal title
\renewbibmacro*{urldate}{} % to not display the urldates
\DeclareFieldFormat{titlecase}{#1} % don't touch titles for now, until I can format them properly
\DeclareFieldFormat{edition}{#1 edition}  % not gonna bother translating it for now

\title{Folklore in Multi-Objective Optimisation}
\author[1]{%
	Oliver Bachtler
	\orcidlink{0000-0001-7942-0750}%
	\thanks{oliver.bachtler@rwth-aachen.de}\textsuperscript{,}%
}
\affil[1]{RWTH Aachen, Aachen, Germany}
\date{}
\usepackage[short]{optidef}

\newcommand{\functDef}[2]{%
	\ifthenelse{\isempty{#2}}{%
		\ensuremath{#1}%
	}{%
		\ensuremath{#1{\left(#2\right)}}%
	}%
}

\newcommand{\calA}{\ensuremath{\mathcal{A}}}
\newcommand{\calP}{\ensuremath{\mathcal{P}}}

\newcommand{\feasibleset}{\ensuremath{X}}
\newcommand{\outcomeset}{\ensuremath{Y}}
\newcommand{\objective}[2][]{\ifthenelse{\isempty{#1}}{\functDef{f}{#2}}{\functDef{f_{#1}}{#2}}}
\newcommand{\objectiveinv}[2][]{\ifthenelse{\isempty{#1}}{\functDef{f^{-1}}{#2}}{\functDef{f^{-1}_{#1}}{#2}}}
\newcommand{\problemp}[1][p]{\ensuremath{\Pi_{#1}}}
\newcommand{\problempws}[2][p]{\ensuremath{\Pi_{#1}^{\text{WS}}(#2)}}

\DeclareMathOperator{\conv}{conv}

\newcommand{\upperimage}{\ensuremath{\outcomeset^\uparrow}}
\newcommand{\leqlex}{\leq_{\text{lex}}}
\DeclareMathOperator{\ext}{ext}
\DeclareMathOperator{\rec}{rec}

\usepackage{todonotes}

\begin{document}

\maketitle

\begin{abstract}
	In this paper, we present and prove some results in multi-objective optimisation that are considered folklore.
	For the most part, proofs for these results exist in special cases, but they are used in more general settings since their proofs can be (largely) transferred.
	We do this transfer explicitly and try to state the results as generally as possible.
	In particular, we also aim at providing clean and complete proofs for results where the original papers are not rigorous.
\end{abstract}

\section{Introduction}

Mathematical folklore is (presumably) present in most branches of mathematics and is probably largely unavoidable in a developing field.
When proving results, one usually cares about the problem one is faced with and does not consider what the minimal assumptions for each statement are.
But if these results, and their proofs, happens to be true in a more general setting, this is usually worth a note at most.
Thus, the more general result is never stated, even if it is accepted as true by the community.

Since I have started looking at multi-objective optimisation, I have encountered a disproportionally large amount of folklore, which can be frustrating at times.
Every once in a while I run into a property that looks too basic to be new, but I cannot find a reference for it.
After I prove it for the case that I care about, someone tells me that this is known in the field and gives me a paper to cite, which does not contain my result precisely, but a similar one.

While I am in the lucky position to have many people around me working in multi-objective optimisation who can tell me whom and what to cite for a result that is \enquote{sufficiently similar}, this is not the case for everyone.
Thus, my goal is to collect such results here in order to make them easier to find (including the references that would have been used before), but to also actually provide a proof for the more general results, which currently do not exist (but the proof for a special case might generalise trivially).
Additionally, I shall provide rigorous proofs for all results, since some of them have (technical) flaws in their original papers.

Since I am starting this collection, its initial form will be biased by my research preferences.
As a result, I want to explicitly encourage any reader who has encountered folklore in their research or studies that is not listed here to contact me, so I can expand this collection.
Thus, I do not know where this paper is heading and I expect its structure to be volatile.
Hence, when referencing specific results, a specification of the version number seems prudent.
But I shall do my best to ensure that a statement present in a previous version can still be found in the newest one, even if the numbering differs.

\paragraph{Contributions and outline.}
Since the objective of this paper is to write down and formally prove results that most people in the field would probably consider known, its main contribution is to find a reference for such results that can easily be used.
Whenever possible, I shall specify the originator of the result, such that they can be cited appropriately as well.
But this paper also does a bit more:
it provides formal proofs for all its results, even if the original reference is not entirely rigorous.
There will also be some new results here, either because they are helpful for the proofs we care about or because they are corollaries of the results we show.

The paper starts with preliminaries, in which I introduce notation needed throughout the paper.
Some first results may appear here as well.
Definitions specific to a subfield will appear in the individual sections where this field is discussed.
Currently, these sections concern the weight set decomposition (with respect to the weighted sum scalarisation) and approximations.
\section{Preliminaries}
\label{sec:prelims}

In this section, we introduce the basic concepts in multi-objective optimisation that we need for the results in the rest of the paper.
A \emph{multi-objective optimisation problem} $\problemp$ is of the following form:
\begin{mini*}{}{\objective{x}=(\objective[1]{x},\objective[2]{x},\ldots,\objective[p]{x})\T}{}{}\tag{\problemp}
    \addConstraint{x}{\in \feasibleset} 
\end{mini*}
where $\objective[1]{},\objective[2]{},\ldots,\objective[p]{}\colon \feasibleset \rightarrow \RR$ are the objective functions and $\feasibleset$ is the feasible set.
Let $\outcomeset=\objective{\feasibleset}$ be the outcome set.
We explicitly make no assumptions on  $\feasibleset$ or $\objective{}$ here.

In this paper, relations like $\leq$ and $<$ are always interpreted component-wise and we write $y \lneq y'$ if $y\leq y'$ but $y\neq y'$.
We use the Pareto-concept of optimality, in which a point $y$ is called \emph{dominated} by $y'$ if and only if $y' \lneq y$. 
A point $y\in \outcomeset$ which is not dominated by any other point is called a \emph{non-dominated} point and the set of non-dominated points is denoted by $\outcomeset_N$.  
We define $\feasibleset_E \coloneqq \objectiveinv{\outcomeset_N}$ as the set of efficient solutions.

A typical task in a multi-objective optimisation problem is to find all non-dominated points in the outcome set $\outcomeset$ and to every non-dominated point $y$ at least one efficient solution $x \in \feasibleset_E$ with $\objective{x} = y$.
This is typically done by the $\varepsilon$-constraint approach or by Tchebycheff scalarisation, which transfer the multi-objective optimisation problem to a single objective one.
Another typical scalarisation that cannot yield all non-dominated points in general is the \emph{weighted-sum scalarisation}, which minimises $w\T\objective{x}$ for $w\gneq 0$.
By normalising $w$, we can always assume that $e\T w = 1$, where $e$ is the all-ones vector.

We let $\Lambda\coloneqq \Lambda_p \coloneqq \Set{\lambda\geq 0 : e\T\lambda = 1}$ be the set of these weights and, for $\lambda\in\Lambda$, we define $\problempws{\lambda}$ by
\begin{mini*}{}{\lambda\T y}{}{}\tag{\problempws{\lambda}}
    \addConstraint{y}{\in \outcomeset}{.}
\end{mini*}
We write $\outcomeset_\lambda$ for all optimal solutions of $\problempws{\lambda}$.
Note that $\min\Set{\lambda\T f(x) : x\in\feasibleset}$ is an equivalent formulation of $\problempws{\lambda}$ in case we care about the preimages.
Thus, we define $\feasibleset_{\lambda}$ as the set of optimal solutions to this problem, so $\feasibleset_{\lambda} = \objectiveinv{\outcomeset_{\lambda}}$.

Let $x\in\feasibleset_E$ and $y=\objective{x}$, so $y\in\outcomeset_N$.
If $y\in\outcomeset_{\lambda}$, for some $\lambda\in\Lambda$, then $y$ is a \emph{supported non-dominated point} and $x\in \feasibleset_\lambda$ is a \emph{supported efficient solution}.
If $y$ is an extreme point of $\conv\outcomeset$, then $y$ is an \emph{extreme-supported non-dominated point} and $x$ an \emph{extreme-supported efficient solution}.
We write $\outcomeset_{SN}$, $\outcomeset_{ESN}$, $\feasibleset_{SE}$, and $\feasibleset_{ESE}$ for the sets of such solutions and points.
We note that several conflicting definitions of supported efficient solutions and supported non-dominated points are used in the literature and we refer to~\cite{Chl25,KS25} for an overview.

A common equivalent definition for $\outcomeset_{ESN}$ is that this is the set of extreme points of the \emph{upper image}, which is the convex set $\upperimage\coloneqq\conv{Y}+\RR_{\geq 0}^p$.
This result appears, for example, in \cite{Ben98,DS92} for linear programming problems, but is just generally true, which we briefly check.
\begin{lemma}
	\label{prelims-characterisation-yesn}
	A point $y$ is in $\outcomeset_{ESN}$ if and only if it is an extreme point of $\upperimage$.
	Additionally, the non-dominated points of $\conv \outcomeset$ and $\upperimage$ coincide.
\end{lemma}
\begin{proof}
	Let $y = \mu y^1 + (1-\mu)y^2$ for $y^1,\, y^2\in \upperimage$ and $\mu\in (0,1)$ be a convex combination of two points in~$\upperimage$.
	Let $y^i \coloneqq z^i + r^i$ with $z^i\in\conv\outcomeset$ and $r^i\geq 0$.
	Then $y = z + r$ where $z \coloneqq \mu z^1 + (1-\mu)z^2$ is a convex combination of elements in $\conv\outcomeset$ and $r\coloneqq \mu r^1 + (1-\mu)r^2\geq 0$.
	
	Let $y$ be a non-dominated extreme point of $\conv\outcomeset$.
	Then, if $y$ is written as the convex combination of two points in $\upperimage$ as above, we get $y=z$ by its non-dominance since $z\in\conv\outcomeset$.
	Since $y$ is extreme and both $z_1$ and $z_2$ are in $\conv\outcomeset$, $z^1 = z^2 = y$, and $y$ is an extreme point of~$\upperimage$.
	
	Conversely, let $y$ be an extreme point of $\upperimage$, then $y\in\conv\outcomeset$ since if $y=z+r$ with $z\in\conv\outcomeset$ and $r\geq 0$, then $y = \tfrac{1}{2}(z + 0) + \tfrac{1}{2}(z + 2r)$, which yields $r=0$.
	Consequently, $y$ is an extreme point of~$\conv\outcomeset$ since this is a subset of $\upperimage$.
	To see that $y$ is non-dominated, assume that $y'\in\conv\outcomeset$ satisfies $y'\leq y$.
	Then $y - y'\geq 0$ and $2y - y' = y + (y-y') \in\upperimage$.
	Since $y = \tfrac{1}{2} y' + \tfrac{1}{2} (2y-y')$, $y'=y$ and $y$ is non-dominated.
	
	For the additional part we simply note that every point in $\upperimage\setminus\conv\outcomeset$ is dominated by one in $\conv\outcomeset$ and, thus, none of them dominate a point in $(\conv\outcomeset)_N$.
\end{proof}

We end the preliminaries with one more common notion of optimality, namely lexicographic optima.
For points $y,\, y'\in\outcomeset$ we define
\begin{displaymath}
	y\leqlex y' \iff y = y' \text{ or } y_i < y_i' \text{ where } i\coloneqq \min\Set{j : y_j\neq y_j'}. 
\end{displaymath}
Then a point $y$ is lexicographically optimal if $y\leqlex y'$ for all $y'\in\outcomeset$.
We can generalise this by first permuting the objectives:
let $\sigma$ be a permutation on $\Set{1,\ldots,p}$ and $\bar{\sigma}\colon \RR^p \to \RR^p$, $y \mapsto (y_{\sigma(1)},\ldots,y_{\sigma(p)})\T$.
We call $y$ lexicographically optimal for $\sigma$ if $\bar{\sigma}(y) \leqlex \bar{\sigma}(y')$ for all $y'\in\outcomeset$.
Additionally, we call $x\in\feasibleset$ lexicographically optimal (for $\sigma$) if and only if $\objective{x}$ is.

Textbooks like \cite{Ehr05,Mie98} typically show that if $x$ is lexicographically optimal, then $x\in \feasibleset_E$, but we would like the result that $x\in\feasibleset_{ESE}$, which is stronger and also true.
To use this result, \cite{HR94} has been used as a citation, but they only show (in Theorem 4.1 and the subsequent corollary) that lexicographically optimal solutions are supported, and only in the case of spanning trees.
For polyhedral sets, a proof can be found in \cite[Lemma~5.9]{Boe18}, but the result is true in general and simple to prove, so let us do just that.
\begin{lemma}
	\label{prelims-lexicographic-ese}
	Let $x$ be lexicographically optimal for a permutation $\sigma$, then $x\in \feasibleset_{ESE}$.
\end{lemma}
\begin{proof}
	Let $y\coloneqq \objective{x}$.
	We need to show that $y\in\outcomeset_{ESN}$ and we may assume that $\sigma$ is the identity.
	Using the citable results, for example \cite[Lemma~5.2]{Ehr05}, we get $y\in\outcomeset_N$ and only need to show that $y$ is an extreme point of $\conv\outcomeset$.
	So let $y = \lambda y^1 + (1-\lambda) y^2$ for $y^1,\, y^2\in\conv\outcomeset$ and $\lambda\in (0,1)$.
	Let $E_i = \Set{y' : y_1' = y_1, \ldots, y_i' = y_i}$, for $i\in\Set{1,\ldots,p}$.
	We now show by induction on $i$ that $y_i=y^1_i=y^2_i$ and $(\conv\outcomeset)\cap E_i = \conv(\outcomeset\cap E_i)$.
	
	Since $y$ is lexicographically optimal, $y_1\leq y_1'$ for all $y'\in \outcomeset$.
	This implies that $y_1\leq y_1'$ for all $y'\in\conv\outcomeset$ and we can conclude that $y_1 \leq y^1_1$ and $y_1\leq y^2_1$.
	Thus, $y_1 = y^1_1 = y^2_1$.
	Moreover, a point $y'\in(\conv\outcomeset)\cap E_1$ is optimal for the first objective, so it is a convex combination of points optimal for the first objective, putting it in $\conv(\outcomeset\cap E_1)$.
	The other inclusion is easy.
	
	The induction step is essentially identical.
	Assume the claim holds up to $i-1$ and regard $i$.
	Now $y_i\leq y_i'$, for all $y'\in \outcomeset\cap E_{i-1}$, meaning this holds for the convex hull of this set as well, which is equal to $(\conv\outcomeset)\cap E_{i-1}$ by induction.
	In particular, the estimation applies to $y^1$ and $y^2$, so $y_i = y^1_i = y^2_i$.
	Finally, a point $y'\in(\conv\outcomeset)\cap E_i$ is in $\conv(\outcomeset\cap E_{i-1})$ by induction, so a convex combination of points that coincide with $y$ in the first $i-1$ components.
	Since $y$ is lexicographically optimal, all these points must also have $y_i$ as their $i$th component, giving us $y'\in\conv(\outcomeset\cap E_i)$.
\end{proof}
\section{Weight Set Decomposition}

In this section, we look at results concerning the weight set decomposition of $\problemp$.
We have already introduced the weight set $\Lambda = \Set{\lambda\geq 0 : e\T\lambda = 1}$.
The weight set component $\Lambda(y)$ of an element~$y\in\outcomeset$ consists of those weights~$\lambda\in\Lambda$ for which $y$ is optimal for $\problempws{\lambda}$, so
\begin{displaymath}
	\Lambda(y) \coloneqq \Set{\lambda\in\Lambda : y \in \outcomeset_{\lambda}}.
\end{displaymath}
Despite the fact that we need strictly positive weights to ensure that optimal solutions of weighted sum problems are non-dominated, the weight set explicitly needs to contain the weights that can lead to dominated images.
The reason is that we want $\Lambda$ to be a polytope.

We also note that we can use $\outcomeset$ and $\upperimage$ interchangeably when it comes to the weight set decomposition.
\begin{observation}
	\label{wsd-outcome-set-equivalent-upper-image}
	Let $\lambda\in\Lambda$.
	Then $\outcomeset_\lambda = \upperimage_\lambda\cap \outcomeset$ and $\upperimage_\lambda \subseteq \conv\outcomeset_{\lambda}+\RR_{\geq 0}^p$.
	In particular, $\outcomeset_{\lambda}\neq\emptyset$ if and only if $\upperimage_\lambda\neq \emptyset$ and for $y\in\outcomeset$ we have $y\in\outcomeset_{\lambda}$ if and only if $y\in\upperimage_\lambda$.
\end{observation}
\begin{proof}
	First, we note that, for $\lambda\in\Lambda$,
	\begin{displaymath}
		\inf\Set{\lambda\T y: y\in\outcomeset} = \inf\Set{\lambda\T y : y\in \conv \outcomeset} = \inf\Set{\lambda\T y : y\in \upperimage}.
	\end{displaymath}
	From this we can deduce that $\outcomeset_\lambda = \upperimage_\lambda\cap \outcomeset$.
	
	Moreover, a point $y\in\upperimage_\lambda$ can be written as a convex combinations of points in $\outcomeset$ plus a non-negative vector.
	Thus, by its optimality, each of the points that appears in the convex combination must also be optimal, placing $y$ in $\conv\outcomeset_{\lambda}+\RR_{\geq 0}^p$.
\end{proof}

The following results are a generalisation of the results in Section~3 of \cite{PGE10}, which we prove rigorously.
We note that this paper contains a few technical flaws (which do not harm the validity of their statements), but might explain why our coverage differs slightly.
Additionally, we prove some results not mentioned there, but that are direct consequences of the properties and may be useful.

The authors of \cite{PGE10} assume that they are given an integer linear program.
In particular, their feasible set $\feasibleset$ is given by $\Set{x\in\ZZ^n : Ax\leq b, x\geq 0}$.
From this they deduce that $\feasibleset$ and $\outcomeset$ are discrete and $\conv \outcomeset $ is a polyhedron.
This need not be true, 
and the authors of \cite{BPST24} fix this problem by assuming rational entries (which guarantees that the integer hull is again a polyhedron).
Alternatively, we could assume that the sets $\feasibleset$ or $\outcomeset$ are finite.
Instead, we make the following weaker assumptions.
\begin{assumption}
	\label{ass:wsd-solutions-closed}
	Let $\problempws{\lambda}$ have an optimal solution for all $\lambda\in\Lambda$ and $\upperimage$ be closed.
\end{assumption}
We wish to make a few remarks concerning this assumption.
The first condition is very natural, since the weight set decomposition can only decompose the weight set if all weights yield an optimal solution.
The requirement on the upper image is one of necessity:
almost none of the results we want to prove hold without it.
The reason is that we want to highlight the use of the non-dominated extreme points, so the extreme points of $\upperimage$, and without assuming that $\upperimage$ is closed, these need not even exist.

However, if $\upperimage$ is closed, we can obtain the following helpful representation for $\upperimage$.
For the proof, we need the following result from convex analysis.
\begin{theoremcite}[{{\cite[Theorem~18.5]{Roc70}}}]
	\label{wsd-conv-set-representation}
	Let $C$ be a closed convex set that is pointed.
	Then $C = \conv(\ext C) + \rec C$.
\end{theoremcite}
In the statement, $\ext C$ denotes the extreme points of $C$ and $\rec C$ is the recession cone of $C$, which consists of those directions $r$ for which $x+\lambda r\in C$ for all $x\in C$ and $\lambda\geq 0$.
The set $C$ is pointed if $(\rec C)\cap(-\rec C) = \Set{0}$.

\begin{lemma}
	\label{wsd-upper-image-rep}
	$\upperimage = \conv\outcomeset_{ESN} + \RR^p_{\geq 0}$.
\end{lemma}
\begin{proof}
	Note that $\upperimage$ satisfies the assumptions of \cref{wsd-conv-set-representation}:
	$\upperimage$ is closed by assumption and convex by definition.
	Since the weighted sum scalarisations have optimal solutions, $\outcomeset$ is bounded from below and non-empty.
	Thus, any $r\in\rec\upperimage$ must be non-negative and all vectors in~$\RR_{\geq 0}^p$ are recession directions, meaning $\rec\upperimage = \RR_{\geq 0}^p$.
	Hence, $\upperimage$ is pointed and, by \cref{wsd-conv-set-representation,prelims-characterisation-yesn}, 
	\begin{displaymath}
		\upperimage = \conv(\ext \upperimage) + \rec \upperimage = \conv\outcomeset_{ESN} + \RR^p_{\geq 0}.
	\end{displaymath}
\end{proof}

Using this representation, we can obtain a useful sufficient condition for verifying \cref{ass:wsd-solutions-closed}.
\begin{lemma}
	\label{wsd-sufficient-criterion-for-assumption}
	Let $\upperimage$ be non-empty, bounded from below, and closed.
	Moreover let $\outcomeset_{ESN}$ be bounded, then \cref{ass:wsd-solutions-closed} holds.
\end{lemma}
\begin{proof}
	We need to show that all weighted sum scalarisations have an optimal solution, that is, $\outcomeset_\lambda\neq\emptyset$ for all $\lambda\in\Lambda$.
	By \cref{wsd-outcome-set-equivalent-upper-image}, it suffices to show that $\upperimage_\lambda\neq \emptyset$ for all $\lambda\in\Lambda$.
	
	Let $\lambda\in\Lambda$.
	By \cref{wsd-upper-image-rep}, $\upperimage=\conv\outcomeset_{ESN} + \RR^p_{\geq 0}$, giving us
	\begin{displaymath}
		\inf \Set{\lambda\T y : y\in \upperimage} = \inf \Set{\lambda\T y : y\in \conv\outcomeset_{ESN}}.
	\end{displaymath}
	Since $\outcomeset_{ESN}$ is bounded, an $M\in\RR$ exists such that $\outcomeset_{ESN} \subseteq \Set{y : y\leq Me}\eqqcolon B$, where $e$ is the all-ones vector.
	Hence, $\upperimage\cap B$ is compact and the problem $\inf\Set{\lambda\T y : y\in\upperimage\cap B}$ has an optimal solution.
	But since $\conv\outcomeset_{ESN}\subseteq \upperimage\cap B\subseteq \upperimage$, $\inf\Set{\lambda\T y : y\in\upperimage\cap B} = \inf\Set{\lambda\T y : y\in \upperimage}$ and $\upperimage_\lambda\neq\emptyset$.
\end{proof}

With our remarks out of the way, let us return to the properties we want to prove.
By \cref{ass:wsd-solutions-closed} all weighted sum scalarisations have optimal solutions and, therefore,
\begin{displaymath}
	\Lambda = \bigcup_{y\in \outcomeset} \Lambda(y).
\end{displaymath}
Let us first look at which of these points are actually required.
If $y\in \outcomeset$ is dominated by a point $y'\in \outcomeset$, then $\lambda\T y \geq \lambda\T y'$ for any weight $\lambda\in\Lambda$.
In particular, $\Lambda(y)\subseteq \Lambda(y')$, since if $y$ is optimal for $\lambda$, then $y'$ must be as well.
Thus, only non-dominated points of $\outcomeset$ can be of interest.
If $y$ is not supported, then $\Lambda(y) = \emptyset$, so such points are not relevant either.
Thus, we can conclude:
\begin{displaymath}
	\Lambda = \bigcup_{y\in \outcomeset_{SN}} \Lambda(y).
\end{displaymath}

This is as far as we can get without using that $\upperimage$ is closed, since the following results rely on the representation of $\upperimage$ from \cref{wsd-upper-image-rep}.
\begin{observation}
	\label{wsd-ysn-subsetet-conv-yesn}
	$\outcomeset_{SN} 
	%\subseteq \upperimage_N
	\subseteq \conv\outcomeset_{ESN}$.
\end{observation}
\begin{proof}
	We have, by \cref{wsd-outcome-set-equivalent-upper-image},
	\begin{displaymath}
		\outcomeset_{SN} 
		= \outcomeset_N \cap \bigcup_{\lambda\in\Lambda} \outcomeset_\lambda 
		= \outcomeset_N \cap \bigcup_{\lambda\in\Lambda} (\upperimage_\lambda \cap\outcomeset)
		= \outcomeset_N \cap \upperimage_{SN}.
	\end{displaymath}
	 By \cite[Corollary~3.7]{Ehr05}, $\upperimage_{SN} = \upperimage_{N}$, so $\outcomeset_{SN}\subseteq \upperimage_N$.
	 Finally, by \cref{wsd-upper-image-rep}, $\upperimage = \conv\outcomeset_{ESN} + \RR^p_{\geq 0}$, so $\upperimage_N\subseteq \conv\outcomeset_{ESN}$.
\end{proof}

We can now show that the weight sets of the extreme-supported non-dominated points are sufficient by showing that the others are subsets.
\begin{lemma}[{{see \cite[Lemma~1]{PGE10}}}]
	Let $y\in\outcomeset_{SN}$.
	Then $y=\sum_{i=1}^q \lambda_iy^i$ is a convex combination of points $\Set{y^1,\ldots y^q}\subseteq\outcomeset_{ESN}$ with $\lambda_1,\ldots,\lambda_q> 0$ and 
	\begin{displaymath}
		\Lambda(y) = \bigcap_{i=1}^q \Lambda(y^i).
	\end{displaymath}
\end{lemma}
\begin{proof}
	The first part follows directly from \cref{wsd-ysn-subsetet-conv-yesn}.
	For the second part, we note that $\lambda\in\Lambda(y)$ implies that $y$ is optimal for $\problempws{\lambda}$.
	But this is the case if and only if all $y^i$ are optimal for $\problempws{\lambda}$, so if $\lambda\in \bigcap_{i=1}^q \Lambda(y^i)$.
\end{proof}

\begin{corollary}[see {{\cite[Proposition~6]{PGE10}}}]
	\label{wsd-yesn-sufficient}
	\begin{displaymath}
		\Lambda = \bigcup_{y\in\outcomeset_{ESN}} \Lambda(y).
	\end{displaymath}
\end{corollary}

We can also provide a description of the weight set components.
\begin{lemma}[see {{\cite[Proposition~2]{PGE10}}}]
	\label{wsd-wsc-description}
	The weight set component $\Lambda(y)$ is given by
	\begin{displaymath}
		\Lambda(y) = \Set{\lambda\in\Lambda : \lambda\T y \leq \lambda\T y' \text{ for all } y'\in \outcomeset_{ESN}}
	\end{displaymath}
	and thus a convex and compact set for all $y\in \outcomeset$.
\end{lemma}
\begin{proof}
	Note that, for $y\in \outcomeset$, $\Lambda(y) = \Set{\lambda\in\Lambda : \lambda\T y \leq \lambda\T y' \text{ for all $y'\in \outcomeset$}}$.
	Since, for a given $\lambda$, the inequalities for the points in $\outcomeset_{\lambda}$ are the most strict, we can omit all constraints that are not in one of these sets.
	Thus, we may restrict $y'$ to the set $\outcomeset_{SN}$, which all live in $\conv \outcomeset_{ESN}$ by \cref{wsd-ysn-subsetet-conv-yesn}.
	But, if we have the inequalities for all extreme points in $\outcomeset_{ESN}$, then the inequalities for convex combinations of these also hold.
	Hence, $\Lambda(y) = \Set{\lambda\in\Lambda : \lambda\T y \leq \lambda\T y' \text{ for all } y'\in \outcomeset_{ESN}}$.
	Since the intersection of convex and closed sets is convex and closed again, the claim follows.
\end{proof}

Now let us take a look at the intersection of two weight set components.
\begin{theorem}[see {{\cite[Proposition~5]{PGE10}}}]
	\label{wsd-intersection-of-wscs-face}
	For $y,\, y'\in \outcomeset$, $\Lambda(y)\cap\Lambda(y')$ is the common face of $\Lambda(y)$ and $\Lambda(y')$ of maximal dimension.
\end{theorem}
\begin{proof}
	Let $H\coloneqq\Set{\lambda : (y-y')\T\lambda = 0}$ be the hyperplane containing those weights for which $y$ and $y'$ are equally good.
	Then $\Lambda(y)\cap\Lambda(y') \subseteq H\cap\Lambda$, since the weights in the intersection are optimal for both $y$ and $y'$, meaning the two points have the same objective value.
	The inequality $\lambda\T y \leq \lambda\T y'$ is valid for~$\Lambda(y)$, so $\Lambda(y)\cap H$ is an exposed face of $\Lambda(y)$ (or empty).
	The same goes for $\Lambda(y')\cap H$.
	
	We observe that $\lambda\in \Lambda(y)\cap H$ implies that $y$, and thus $y'$, is optimal for $\lambda$, so $\lambda\in\Lambda(y')$.
	This also works if we swap $y$ and $y'$, showing that $\Lambda(y)\cap H = \Lambda(y')\cap H$.
	Hence,
	\begin{displaymath}
		\Lambda(y)\cap\Lambda(y') = \Lambda(y)\cap\Lambda(y')\cap H = \Lambda(y)\cap\Lambda(y)\cap H = \Lambda(y)\cap H.
	\end{displaymath}
	So the intersection is a common face of $\Lambda(y)$ and $\Lambda(y')$.
	
	The part about maximal dimension is no real claim:
	since we showed that the intersection is a common face, it is the largest common face, so the one of maximal dimension.
	It simply serves to specify which common face the intersection is.
\end{proof}

From this we can obtain a result that concerns $\feasibleset_{SE}$.
By definition
\begin{displaymath}
	\feasibleset_{SE} = \feasibleset_E\cap \bigcup_{\lambda\in\Lambda} \feasibleset_\lambda.
\end{displaymath}
For $y\in\outcomeset_{ESN}$ we know that $\Lambda(y)$ is is a compact convex set by \cref{wsd-wsc-description}.
Thus, it is the convex hull of its extreme points by \cite[Corollary~18.5.1]{Roc70}.
If we just solve the weighted sums for these weights, we miss nothing.
\begin{theorem}
	\label{wsd-xse-decomposition}
	Let $S\coloneqq \bigcup_{y\in\outcomeset_{ESN}} \ext\Lambda(y)$, then
	\begin{displaymath}
		\feasibleset_{SE} = \feasibleset_E \cap \bigcup_{\lambda\in S} \feasibleset_\lambda.
	\end{displaymath}
\end{theorem}
\begin{proof}
	The inequality $\feasibleset_E \cap \bigcup_{\lambda\in S} \feasibleset_\lambda\subseteq \feasibleset_{SE}$ holds since we just restricted the weights we may use.
	Thus, let $x\in\feasibleset_{SE}$, so $x\in\feasibleset_E\cap\feasibleset_{\lambda}$ for some $\lambda\in\Lambda$.
	By \cref{wsd-yesn-sufficient}, $\lambda\in\Lambda(y)$ for some $y\in\outcomeset_{ESN}$ and $F\coloneqq\Lambda(\objective{x})\cap\Lambda(y)\neq\emptyset$.
	By \cref{wsd-intersection-of-wscs-face}, $F$ is a face of $\Lambda(y)$, so it contains an extreme point $\bar{\lambda}\in\ext\Lambda(y)$.
	Hence, $\bar{\lambda}\in\Lambda(\objective{x})$ and $x\in\feasibleset_{\bar{\lambda}}$.
\end{proof}

For the next results, which rely on polyhedral theory, we need that we only have finitely many extreme points, which we now assume.
\begin{assumption}
	\label{ass:wsd-fin-eps}
	Let $\outcomeset_{ESN}$ be finite, say $\outcomeset_{ESN} = \Set{y^1,\ldots,y^q}$.
\end{assumption}

As a direct corollary of \cref{wsd-upper-image-rep}, we obtain that $\upperimage$ is a polyhedron.
Similarly, as a corollary to \cref{wsd-wsc-description}, we get:
\begin{corollary}[see {{\cite[Proposition~3]{PGE10}}}]
	\label{wsd-wsc-polytopes}
	For $y\in\outcomeset$, $\Lambda(y)$ is a polytope.
\end{corollary}

It turns out that in this polyhedral setting, we cannot discard any more points in the representation of the weight set from \cref{wsd-yesn-sufficient} since each extreme-supported non-dominated point is the unique optimal solution for some weight.
This is, in fact, a classification of the points in $\outcomeset_{ESN}$.
\begin{lemma}
	\label{wsd-wsc-for-yesn-unique-weight}
	Let $y\in \outcomeset$.
	Then $y\in \outcomeset_{ESN}$ if and only if there exists a $0<\lambda\in\Lambda$ with 
	\begin{displaymath}
		\lambda\in\Lambda(y)\setminus \bigcup_{y'\in \outcomeset: y'\neq y} \Lambda(y').
	\end{displaymath}
\end{lemma}
\begin{proof}
	Note that, by polyhedral theory~\cite[Theorem~4.6]{NW88ch4}, $y\in \outcomeset_{ESN}$ is an extreme point of $\upperimage$ if and only if $y$ is the unique optimal solution (in $\upperimage$) for some linear objective $\lambda$.
	Since $\lambda$ cannot have negative entries without making the problem unbounded, we get that $\lambda\geq 0$.
	If $\lambda$ has a zero-entry, the solution $y$ would not be unique, so, in fact, $\lambda>0$ and we may assume that $\lambda\in\Lambda$.
	This shows that $y\in \outcomeset_{ESN}$ if and only if there exists a $0<\lambda\in\Lambda$ for which $y$ is the unique optimal solution.
	In particular, this is equivalent to $\lambda\notin \Lambda(y')$ for any other $y'\in \outcomeset$.
\end{proof}

\begin{corollary}[see {{\cite[Corollary~1]{PGE10}}}]
	\label{wsd-yesn-necessary}
	If $\Lambda = \bigcup_{y\in S} \Lambda(y)$ for $S\subseteq \outcomeset$, then $\outcomeset_{ESN}\subseteq S$.
\end{corollary}

Another classification of the points in $\outcomeset_{ESN}$ is that they are exactly those points that yield full-dimensional polytopes.
\begin{theorem}[see {{\cite[Proposition~6]{PGE10}}}]
	\label{wsd-yesn-iff-wsc-full-dim}
	For $y\in \outcomeset$, $\Lambda(y)$ has dimension $p-1$ if and only if $y\in \outcomeset_{ESN}$.
\end{theorem}
\begin{proof}
	Recall that a $\Lambda(y)$ is a polytope by \cref{wsd-wsc-polytopes} and it has dimension $p-1$ if and only if its equality matrix has rank~$1$ by \cite[Propositions~2.4]{NW88ch4}. 
	Also by \cref{wsd-wsc-description},
	\begin{displaymath}
		\Lambda(y) = \Set{\lambda: \lambda\geq 0, e\T \lambda=1, (y-y^1)\T\lambda \leq 0,\ldots,(y-y^q)\T\lambda\leq 0}.
	\end{displaymath}
	
	If $y\in \outcomeset_{ESN}$, we get some weight $0<\lambda\in\Lambda(y)$ that satisfies $\lambda\T y' > \lambda\T y$ for all $y\neq y'\in \outcomeset$ by \cref{wsd-wsc-for-yesn-unique-weight}.
	Thus, the only inequalities of $\Lambda(y)$ that are satisfied with equality for $\lambda$ are $e\T \lambda = 1$ and $(y-y)\T\lambda = 0$.
	These are clearly in the equality set and the equality matrix has rank~1.
	
	For the converse, assume that $\Lambda(y)$ has dimension $p-1$ and let $\hat{\lambda}$ be an inner point of $\Lambda(y)$.
	Then, since the equality matrix has rank $1$, all constraints that $\hat{\lambda}$ satisfies with equality are scalar multiples of $e\T\lambda=1$.
	Note that all other right hand sides are $0$, so such a constraint can only be of the form $o\T\lambda = 0$, which means that $y=y^i$ for some $i$.
	In this case, $y\in \outcomeset_{ESN}$ as desired.
	The only remaining alternative, that $\hat{\lambda}\T y < \hat{\lambda}\T y^i$ for all $i$ is not possible since $y\in\outcomeset\subseteq\upperimage$, so $y$ can be written as a convex combination of the points in $\outcomeset_{ESN}$ plus a non-negative vector by \cref{wsd-upper-image-rep}.
	In particular, at least one of the extreme-supported non-dominated points must be at least as good.
\end{proof}
\section{Approximation}

In this section, we shall show that the supported efficient solutions are, in fact, a $\Set{(2,1),(1,2)}$-approximation for bi-objective problems (and a bit more).
Everything required for this proof can already be found in \cite{BRTV21}, though the authors only obtain a $\Set{(2+\varepsilon,1),(1,2+\varepsilon)}$-approximation.
But the $\varepsilon$ that appears in their proof is allowed to be zero, as we shall verify here.

Before we do, let us briefly recap them here.
We need to make the typical assumption for approximation algorithms here, namely that $\outcomeset\subseteq \RR^p_{>0}$.
We also note that the result specifically targets minimisation problems and does not hold for maximisation ones (see Appendix~B in \cite{BRTV21}).

Now let us state the definitions:
in single-objective optimisation, a feasible solution $x\in \feasibleset$ is an $\alpha$-approximation ($\alpha\geq 1$) of a solution $x'$ if $\objective{x} \leq \alpha\cdot\objective{x'}$.
If $x$ $\alpha$-approximates every other feasible solution of \problemp[1], it is an $\alpha$-approximation of \problemp[1].
A straightforward generalisation to the multi-objective case is the following:
let $\alpha\in\RR^p_{\geq 1}$.
A feasible solution $x\in \feasibleset$ $\alpha$-approximates another feasible solution $x' \in \feasibleset$ if $\objective[i]{x} \leq \alpha_i \cdot \objective[i]{x'}$ for all $i=1,\ldots,p$.
Again, if $x$ $\alpha$-approximates every feasible solution of~\problemp, it is called an $\alpha$-approximation of \problemp.

In multi-objective optimisation, this is very restrictive however, since we do not have a single optimal solution, but a set.
Hence, it makes sense to allow us to approximate these solutions by a set as well.
\begin{definition}
	Let $\alpha\in\RR^p_{\geq 1}$.
	A set $\calP \subseteq \feasibleset$ is an \emph{$\alpha$-approximation} of $\problemp$ if every feasible solution $x' \in \feasibleset$ is $\alpha$-approximated by a solution $x \in \calP$.
\end{definition}

To generalise these natural definitions, \textcite{BRTV21} introduced the following version, which allows a set of approximation factors. 
\begin{definition}
	Let $\calA \subseteq \RR^p_{\geq 1}$ be a set of approximation factors. 
	A set $\calP \subseteq \feasibleset$ is an \emph{$\calA$-approximation} for $\problemp$ if every feasible solution $x' \in \feasibleset$ is $\alpha$-approximated by a feasible solution $x \in \calP$ for some $\alpha \in \calA$.
\end{definition}

Now that we have recalled the definitions, we can prove the approximation result, which requires the following two lemmas.
\begin{lemma}
	\label{approx-first-approximation}
	Let $x \in \feasibleset$ be an optimal solution of $\problempws{\lambda}$ for some $\lambda\gneq 0$. 
	For every $x' \in \feasibleset$, there exists a $j \in \Set{1,\ldots,p}$ such that $x$ $1$-approximates $x'$ with respect to $\objective[j]{}$, that is,
	\begin{displaymath}
		\objective[j]{x} \leq \objective[j]{x'} \text{ for some } j \in \Set{1,\ldots,p}.
	\end{displaymath}    
\end{lemma}
\begin{proof}
	Let $x' \in \feasibleset$ and suppose that $\objective[i]{x'} < \objective[i]{x}$ for all $i\in\Set{1,\ldots,p}$. 
	Then $\sum_{i=1}^{p} \lambda_i \cdot \objective[i]{x'} < \sum_{i=1}^{p} \lambda_i \cdot \objective[i]{x}$, contradicting the optimality of $x$.
\end{proof}

\begin{lemma}
	\label{approx-ws-approximation}
	Let $x' \in \feasibleset$ and $\lambda$ be given by $\lambda_i = \tfrac{1}{\objective[i]{x'}}>0$ for $i\in\Set{1,\ldots,p}$.
	If $x$ is optimal for $\problempws{\lambda}$, then $x$ $\alpha$-approximates $x'$ for some $\alpha \in \RR_{\geq 1}^p$ with
	\begin{displaymath}
		\sum_{i\colon \alpha_i>1}\alpha_i \leq p \qquad \text{ and } \qquad \alpha_j = 1 \text{ for at least one $j$}.
	\end{displaymath}
\end{lemma}
\begin{proof}
	Let $\lambda \in \RR^p_{>0}$ as defined above. 
	Since $x$ is optimal for $\problempws{\lambda}$,
	\begin{displaymath}
		\sum_{i=1}^{p}\lambda_i\cdot \objective[i]{x} 
		\leq  \sum_{i=1}^{p}\lambda_i\cdot \objective[i]{x'} 
		= \sum_{i=1}^{p}{\frac{1}{\objective[i]{x'}}\cdot \objective[i]{x'}} 
		= p.
	\end{displaymath}
	By setting $\alpha_i = \max\Set{1,\tfrac{\objective[i]{x}}{\objective[i]{x'}}}$ for $i\in\Set{1,\ldots,p}$, $x$ $\alpha$-approximates $x'$ and
	\begin{displaymath}
		\sum_{i\colon \alpha_i>1}\alpha_i 
		\leq \sum_{i=1}^{p} \frac{1}{\objective[i]{x'}}\cdot \objective[i]{x} 
		= \sum_{i=1}^{p}\lambda_i\cdot \objective[i]{x}
		\leq p.
	\end{displaymath}
	
	By \cref{approx-first-approximation}, $\objective[j]{x} \leq \objective[j]{x'}$ for some $j \in \Set{1,\ldots,p}$ and $\alpha_j = 1$.
\end{proof}

This yields the approximation result as a corollary.
\begin{corollary}
	\label{approx-ese-approximation}
	The set $\feasibleset_{SE}$ is an $\calA$-approximation of all efficient solutions, where
	\begin{displaymath}
		\calA = \Set{(\alpha_1,\ldots,\alpha_p)\setdelimiter \alpha_i \geq 1 \text{ for all } i,\, \alpha_j = 1 \text{ for some $j$, and } \sum_{i\colon \alpha_i>1}{\alpha_i} = p}.
	\end{displaymath}
	In the special case for $p=2$, we obtain a $\Set{(1,2),(2,1)}$-approximation.
\end{corollary}

By \cite[Theorem~2]{BRTV21} we know that this result is tight.
\begin{lemma}
	The approximation result obtained in \cref{approx-ese-approximation} is tight, that is, the result becomes false when replacing $p$ by $p-\varepsilon$ for any positive $\varepsilon$.
\end{lemma}

We end on two final notes:
First, we did not need the entire set set $\feasibleset_{SE}$ to obtain the approximation:
it suffices to have one preimage for each point in $\outcomeset_{SN}$.
Even better, if \cref{ass:wsd-solutions-closed} is satisfied, then we know that the extreme-supported non-dominated points contain an optimal point for every possible weight.
Hence, a preimage for each point in $\outcomeset_{ESN}$ is also sufficient.

With regard to \cref{wsd-sufficient-criterion-for-assumption}, if $\upperimage=\emptyset$ then $\outcomeset_{ESN}=\emptyset$ is an approximation.
Also, we already needed to assume that $\outcomeset\subseteq \RR_{>0}^p$.
Thus, the only additional assumptions are that $\upperimage$ is closed and $\outcomeset_{ESN}$ is bounded.
Let us specify this:
\begin{corollary}
	If \cref{ass:wsd-solutions-closed} holds, for example if $\upperimage$ is closed and $\outcomeset_{ESN}$ is bounded, then any set $S\subseteq\feasibleset$ with $\objective{S}\supseteq \outcomeset_{ESN}$ is \calA-approximation where $\calA$ is defined as in \cref{approx-ese-approximation}.
\end{corollary}
\clearpage
\printbibliography

\clearpage
\section*{Acknowledgements}
Oliver Bachtler was funded by the Deutsche Forschungsgemeinschaft (DFG, German Research Foundation) -- GRK 2982, 516090167 \enquote{Mathematics of Interdisciplinary Multiobjective Optimization} while at the RPTU Kaiserslautern-Landau.

I want to thank the following people that proofread parts of this paper:
Hannah Borgmann, Philipp Hermann, Levin Nemesch, Dorotea Redžepi.

\end{document}